\documentclass[10pt]{amsart}
\usepackage{enumerate,amsmath,amssymb,latexsym,
amsfonts, amsthm, amscd, MnSymbol}

\theoremstyle{plain}

\newtheorem{theorem}{Theorem}

\numberwithin{equation}{section}

\newcommand{\ra}{\rightarrow}

\addtocounter{section}{-1}

\begin{document}

\title {Partial Inner Products on Antiduals}

\date{}

\author[P.L. Robinson]{P.L. Robinson}

\address{Department of Mathematics \\ University of Florida \\ Gainesville FL 32611  USA }

\email[]{paulr@ufl.edu}

 \keywords{}

\begin{abstract}

We discuss extensions of an inner product from a vector space to its full antidual. None of these extensions is weakly continuous, but partial extensions recapture some familiar structure including the Hilbert space completion and the antiduality pairing. 

\end{abstract}

\maketitle

\medbreak

\medbreak 

Let $V$ be an infinite-dimensional complex vector space on which $\langle \bullet | \bullet \rangle$ is an inner product. We adopt the convention according to which $\langle x | y \rangle$ is antilinear in $x$ and linear in $y$; we make no assumption regarding completeness of the inner product. Let $V'$ be the full antidual of $V$: thus, $V'$ comprises precisely all antilinear maps $V \ra \mathbb{C}$ whether bounded or otherwise. The inner product $\langle \bullet | \bullet \rangle$ engenders a canonical linear embedding of $V$ in $V'$: explicitly, for each $v \in V$ we define $v' \in V'$ by the rule that if $z \in V$ is arbitrary then 
$$v' (z) = \langle z | v \rangle.$$

\medbreak 

Our aim is to investigate inner products on $V'$ that are compatible with the given inner product $\langle \bullet | \bullet \rangle$ on $V$. At the very least, we should insist that compatibility requires the embedding $V \ra V'$ to be isometric. Were this our only compatibility requirement, a suitable inner product on $V'$ could of course be defined by transporting the given inner product to $\widehat{V} = \{ v' : v \in V \} \subset V'$ and choosing a (purely algebraic) decomposition $V' = \widehat{V} \oplus W$, providing $W$ with an inner product and making the decomposition orthogonal. 

\medbreak 

We shall demand more of compatibility. The full antidual $V'$ naturally carries the (weak) topology of pointwise convergence, according to which a net $(\zeta_{\delta} : \delta \in \Delta)$ converges to $\zeta$ in $V'$ precisely when $\zeta_{\delta} (v) \ra \zeta (v)$ for every $v \in V$. We shall say that the inner product $[ \bullet | \bullet ]$ on $V'$ is {\it compatible} with the original  inner product $\langle \bullet | \bullet \rangle$ on $V$ precisely when: \par
(i) the canonical embedding $ V \ra V'$ is isometric, so that if $x, y \in V$ then 
$$[x' | y' ] = \langle x | y \rangle;$$\par 
(ii) $[ \bullet | \bullet ]$ is weakly continuous in each slot, so that if $\xi_{\delta} \ra \xi$ and $\eta_{\delta} \ra \eta$  then 
$$[\xi_{\delta} | \eta ] \ra [\xi | \eta ] \; \; \; {\rm and} \; \; \; [\xi | \eta_{\delta}  ] \ra [\xi | \eta].$$

\medbreak 

Our approach to this investigation will be by way of finite-dimensional approximation. Write $\mathcal{F} (V)$ for the set comprising all finite-dimensional complex subspaces of $V$; this set is naturally directed by inclusion. Let $\zeta : V \ra \mathbb{C}$ be an antilinear functional on $V$. If $M \in \mathcal{F} (V)$ is any finite-dimensional subspace of $V$ then the restriction $\zeta |_M : M \ra \mathbb{C}$ is given by taking inner product against a vector in $M$: there exists a unique vector $\zeta_M \in M$ such that if $z \in M$ then 
$$\zeta (z)  = \langle z | \zeta_M \rangle.$$

\medbreak 

\begin{theorem} \label{net} 
If $\zeta \in V'$ then the net $(\zeta_M' : M \in \mathcal{F} (V))$ converges weakly to $\zeta$ in $V'$. 
\end{theorem} 

\begin{proof} 
Let $z \in V$ be arbitrary: on the one hand, if $M \in \mathcal{F} (M)$ then $\zeta_M' (z) = \langle z | \zeta_M \rangle$; on the other, if also $z \in M$ then $\langle z | \zeta_M \rangle = \zeta (z)$. Thus the given net is eventually constant at each point and so pointwise convergent, with the correct limit. 
\end{proof} 

\medbreak 

Otherwise said, if $\zeta \in V'$ then 
$$\zeta = \lim_{M \uparrow \mathcal{F} (V)} \zeta_M'.$$

\medbreak 

We may use these elementary finite-dimensional approximating nets to analyze a compatible inner product. 

\medbreak 

First we note that a compatible inner product in $V'$ restricts to reproduce the natural `duality' pairing between $V$ and its antidual. 

\medbreak 

\begin{theorem} \label{pairing}
Let $[ \bullet | \bullet ]$ be a compatible inner product on $V'$. If $x, y \in V$ and $\xi, \eta \in V'$ then
$$[x' | \eta] = \eta (x), \; \; \; \; \; [\xi | y' ] = \overline{\xi (y)}.$$
\end{theorem} 

\begin{proof} 
We need only establish the first identity. For this, let $M \in \mathcal{F} (V)$: once $ M$ contains $x$ it follows that  
$$\eta (x) = \langle x | \eta_M \rangle = [x' | \eta_M' ] \ra [x' | \eta]$$
on account of compatibility and Theorem \ref{net}. 
\end{proof} 

\medbreak

Next we note that finite-dimensional approximants permit the reconstruction of a compatible inner product in its entirety. 

\medbreak 

\begin{theorem} \label{un} 
Let $[ \bullet | \bullet ]$ be a compatible inner product on $V'$. If $\xi, \eta \in V'$ then 
$$[\xi | \eta ] = \lim_{M \uparrow \mathcal{F} (V)} \langle \xi_M | \eta_M \rangle.$$
\end{theorem} 

\begin{proof} 
If $M \in \mathcal{F} (V)$ then Theorem \ref{pairing} ensures that 
$$\langle \xi_M | \eta_M \rangle = \eta (\xi_M) = [ \xi_M' | \eta ]$$
whereupon compatibility and Theorem \ref{net} complete the argument. 
\end{proof} 

\medbreak

Accordingly, $V'$ carries {\it at most one} compatible inner product; we now address the question of existence. 

\medbreak

It will help to have available the relationship between the approximants relative to a finite-dimensional subspace and one of its hyperplanes. 

\medbreak 

\begin{theorem} \label{hyper}
Let $N = M \oplus \mathbb{C} u$ where $u \in V$ is a unit vector orthogonal to $M \in \mathcal{F} (V)$. If $\zeta \in V'$ then 
$$\zeta_N = \zeta_M + \zeta(u) u.$$
\end{theorem} 

\begin{proof} 
Immediate: if $z \in M$ then $\zeta (z) = \langle z | \zeta_M \rangle$ and if $\lambda \in \mathbb{C}$ then $\zeta (\lambda u) = \bar{\lambda} \zeta (u) = \langle \lambda u | \zeta (u) u \rangle$. 
\end{proof} 

\medbreak 

It follows at once that if $\xi, \eta \in V'$ then 
$$\langle \xi_N | \eta_N \rangle = \langle \xi_M | \eta_M \rangle + \overline{\xi (u)} \eta (u).$$
Inductively, if $N = M \oplus L$ is an orthogonal decomposition then 
$$||\zeta_N ||^2 = || \zeta_M ||^2 + || \zeta_L ||^2;$$
indeed, $\zeta_M$ and $\zeta_L$ are the respective orthogonal projections of $\zeta_N$ on $M$ and $L$.

\medbreak 

The following result will also be useful. 

\medbreak 

\begin{theorem} \label{rad}
If $x, y \in V$ are unit vectors then the supremum of $|\langle x | u \rangle \langle u | y \rangle |$ as $u \in V$ runs over all unit vectors is at least $1/2$. 
\end{theorem} 

\begin{proof} 
Define $T : V \ra V$ by $T(z) = \langle x | z \rangle y$; note that $T$ has unit operator norm. Quite generally, the numerical radius $w(T)$ of the operator $T$ is defined by 
$$w(T) = \sup \{ |\langle u | Tu \rangle| : ||u|| = 1 \};$$
for example, see [1]. Here, 
$$\langle u | T u \rangle = \langle x | u \rangle \langle u | y \rangle$$
so that $w(T)$ is precisely the supremum described in the statement of the theorem. 
If $u, v \in V$ are unit vectors then by polarization 
$$2\langle u | T v \rangle + 2 \langle v | T u \rangle = \langle u + v | T (u + v) \rangle - \langle u - v | T (u - v) \rangle$$
whence the parallelogram law yields
$$2 | \langle u | T v \rangle + 2 \langle v | T u \rangle | \leqslant w(T) \{ || u + v ||^2 + || u - v ||^2\} = 4 w(T);$$
furthermore, if $v = T(u)/||T(u)||$ then 
$$\langle u | T v \rangle +  \langle v | T u \rangle = \frac{\langle u | T^2 u \rangle}{||T(u)||} + ||T(u)||.$$
Choose the unimodular scalar $\lambda$ so that $\lambda^2 \langle u | T^2 u \rangle \geqslant 0$ and apply the foregoing analysis to $\lambda T$ in place of $T$ itself, to deduce that $2 w(T) \geqslant || T u ||$. Finally, take the supremum as $u$ runs over all unit vectors, to conclude that $2 w(T) \geqslant || T || = 1$. 
\end{proof} 

\medbreak

Our present purposes are adequately served by this estimate, but the identification of this supremum as a numerical radius leads to an exact formula. The cleanest formula obtains when $( \bullet | \bullet )$ is a {\it real} inner product, in which case the set of all reals $(x | u) (u | y)$ as $u$ runs over all unit vectors is a closed interval of unit length, namely 
$$[((x | y) - 1)/2 , ((x | y) + 1)/2 ].$$

\medbreak 

Notice that the operator norm of any (bounded) antifunctional on $V$ can be identified in terms of its finite-dimensional approximants. 

\medbreak 

\begin{theorem} \label{norm}
If $\zeta \in V'$ then 
$$|| \zeta || = \sup \{ || \zeta_M || : M \in \mathcal{F} (V) \} \in [0, \infty].$$
\end{theorem} 

\begin{proof} 
In one direction, let $K$ be the indicated supremum: if $z \in V$ then let $M = \mathbb{C} z$ and calculate $| \zeta (z)| = | \langle z | \zeta_M \rangle | \leqslant || z || \; || \zeta_M || \leqslant K || z ||$. In the opposite direction, if $M \in \mathcal{F} (V)$ then $\zeta_M \in M$ so that $|| \zeta_M ||^2 = \langle \zeta_M | \zeta_M \rangle = \zeta (\zeta_M) \leqslant || \zeta || \; || \zeta_M ||$ and cancellation ends the argument. 
\end{proof} 

\medbreak 

In fact, the net $(|| \zeta_M || : M \in \mathcal{F} (V))$ is increasing, as the remark after Theorem \ref{hyper} makes clear; consequently, 
$$|| \zeta || = \lim_{M \uparrow \mathcal{F} (V)} || \zeta_M ||.$$

\medbreak 

We coordinate these theorems to effect a proof of the next. 

\medbreak 

\begin{theorem} \label{non}
If the antifunctionals $\xi, \eta \in V'$ are unbounded then the net 
$( \langle \xi_M | \eta_M \rangle : M \in \mathcal{F} (V))$
does not converge. 
\end{theorem} 

\begin{proof} 
The net $( \langle \xi_M | \eta_M \rangle : M \in \mathcal{F} (V))$ is not Cauchy. In fact, let the finite-dimensional subspace $M \in \mathcal{F} (V)$ be arbitrary. The restrictions of $\xi$ and $\eta$ to the orthocomplement $M^{\perp}$ being unbounded, Theorem \ref{norm} provides $L \in \mathcal{F} (M^{\perp})$  such that $|| \xi_L ||$ and $|| \eta_L ||$ are as large as we please; say greater than unity. Theorem \ref{rad} provides a unit vector $u \in L$ such that  $ \overline{\xi (u)} \eta (u) =  \langle \xi_L | u \rangle \langle u | \eta_L \rangle$ has modulus greater than $1/2$. Finally, Theorem \ref{hyper} shows that $N = M \oplus \mathbb{C} u \in \mathcal{F}(V)$ satisfies 
$| \langle \xi_N | \eta_N \rangle - \langle \xi_M | \eta_M \rangle | > 1/2.$
\end{proof} 

\medbreak

Theorem \ref{un} and Theorem \ref{non} together imply that compatible inner products are {\it nonexistent}. We are led to ask what can be salvaged from this negative result. 

\medbreak 

Taking a cue from Theorem \ref{un} we define the {\it partial} inner product $[ \bullet | \bullet ] = [ \bullet | \bullet ]_V$ in $V'$ by the rule that if $\xi, \eta \in V'$ then 
$$[\xi | \eta] : = \lim_{M \uparrow \mathcal{F} (V)} \langle \xi_M | \eta_M \rangle$$
{\it whenever this limit exists}. 

\medbreak 

This rule does define a partial inner product extending $\langle \bullet | \bullet \rangle$. It is plainly Hermitian, in the sense that if $[\xi | \eta ]$ is defined then so is $[ \eta | \xi ]$ and 
$$[ \eta | \xi ] = \overline{ [ \xi | \eta ] }.$$
It is plainly also linear in the second slot (and therefore antilinear in the first) in the sense that if $[ \xi | \eta_1 ]$ and $[ \xi | \eta_2 ]$ are defined then so is $[ \xi | \lambda_1 \eta_1 + \lambda_2 \eta_2 ]$ and 
$$[ \xi | \lambda_1 \eta_1 + \lambda_2 \eta_2 ] = [ \xi | \eta_1 ] + [ \xi | \eta_2 ]$$
whenever $\lambda_1 , \lambda_2 \in \mathbb{C}$. Finally, Theorem \ref{norm} makes it clear that $[ \zeta | \zeta ]$ is defined precisely when the antifunctional $\zeta$ is bounded, in which case $[ \zeta | \zeta ] \geqslant 0$ with equality if and only if $\zeta = 0$. 

\medbreak 

This last point can be amplified a little. Let us denote by $V^* \subset V'$ the subspace comprising all {\it bounded} antifunctionals. Let $\xi, \eta \in V^*$: polarization in $V$ shows that if $M \in \mathcal{F} (V)$ then 
$$4 \langle \xi_M | \eta_M \rangle = \sum_{n = 0}^3 i^{-n} || \xi_M + i^n \eta_M ||^2 = \sum_{n = 0}^3 i^{-n} || (\xi + i^n \eta)_M ||^2$$
and the remark after Theorem \ref{norm} justifies passage to the limit as $M \uparrow \mathcal{F} (V)$ producing 
$$4 [ \xi | \eta ] =  \sum_{n = 0}^3 i^{-n} || \xi + i^n \eta ||^2$$
with operator norm on the right. Thus the partial inner product $[ \bullet | \bullet ]$ is defined on $V^*$ where it becomes a true inner product underlying the operator norm. More is true: it may be checked (as an instructive exercise) that if $\zeta \in V^*$ then the net $(\zeta_M' : M \in \mathcal{F} (V))$ converges to $\zeta$ in operator norm, improving Theorem \ref{net} in this circumstance; so $V^*$ furnishes a canonical model for the Hilbert space completion of $V$. 

\medbreak 

This partial inner product $[ \bullet | \bullet ]$ is also defined on $V \times V'$ and $V' \times V$ upon which it induces the natural pairing between $V$ and $V'$: an argument akin to the one for Theorem \ref{pairing} shows that if $x, y \in V$ and $\xi, \eta \in V'$ then
$$[x' | \eta] = \eta (x), \; \; \; \; \; [\xi | y' ] = \overline{\xi (y)}.$$

\medbreak

We may extend this partial inner product by replacing the directed set $\mathcal{F} (V)$ with one of its cofinal subsets $\mathcal{S} \subseteq \mathcal{F} (V)$: if the net $( \langle \xi_M | \eta_M \rangle : M \in \mathcal{F} (V) )$ converges then so does its subnet $( \langle \xi_M | \eta_M \rangle : M \in \mathcal{S} )$ and the limits coincide; however, the latter net may converge even though the former does not. Of course, each such extension will continue to reproduce both the Hilbert space completion of $V$ and the canonical pairing with its antidual; but such an extension may have further properties. 

\medbreak 

One example will suffice as an illustration. Let $$V = X \oplus Y$$ be an orthogonal decomposition, with $P_X : V \ra X$ and $P_Y : V \ra Y$ as corresponding orthogonal projectors. Note that antilinear extension by zero on orthocomplements yields the canonical embeddings $X' \ra V' : \xi \mapsto \xi \circ P_X$ and $Y' \ra V' : \eta \mapsto \eta \circ P_Y$. Write $\mathcal{F} (X, Y)$ for the set comprising all finite-dimensional subspaces $M$ of $V$ that split under this decomposition as $M = M_X \oplus M_Y$ where $M_X = P_X (M)$ and $M_Y = P_Y (M)$. The subset $\mathcal{F} (X, Y) \subseteq \mathcal{F} (V)$ is certainly cofinal: indeed, each $M \in \mathcal{F} (V)$ is contained in $M_X \oplus M_Y \in \mathcal{F} (X, Y)$. Now, when $\xi, \eta \in V'$ let us agree to write 
$$[ \xi | \eta ]_{X, Y} = \lim_{M \uparrow \mathcal{F} (X, Y)} \langle \xi_M | \eta_M \rangle.$$ 

\medbreak 

This new partial inner product circumvents Theorem \ref{non} in being defined on certain pairs of unbounded antifunctionals; for example, as in the following result.

\medbreak 

\begin{theorem} \label{XY}
If $\xi \in X'$ and $\eta \in Y'$ then $[\xi \circ  P_X | \eta \circ P_Y ]_{X, Y} = 0$. 
\end{theorem} 

\begin{proof} 
As $X$ and $Y$ are orthogonal, this follows at once from the fact (left as another exercise) that if $M \in \mathcal{F} (X, Y)$ then $(\xi \circ P_X)_M = \xi_{M_X} \in M_X \subseteq X$ and $(\eta \circ P_Y)_M = \eta_{M_Y} \in M_Y \subseteq Y.$
\end{proof} 

\medbreak

It is appropriate here to issue the reminder that if $\xi \in X'$ and $\eta \in Y'$ are unbounded then $[ \xi \circ P_X | \eta \circ P_Y ]_V$ is {\it undefined}: the net $( \langle (\xi \circ P_X)_M | (\eta \circ P_Y)_M \rangle : M \in \mathcal{F} (V) )$ does not converge, even though $\xi \circ P_X$ vanishes on $Y = X^{\perp}$ and $\eta \circ P_Y$ vanishes on $X = Y^{\perp}$.

\medbreak

\bigbreak

\begin{center} 
{\small R}{\footnotesize EFERENCES}
\end{center} 
\medbreak

[1] G.K. Pedersen, {\it Analysis Now}, Springer Graduate Texts in Mathematics 118 (1989). 

\medbreak

\end{document}